\documentclass{article}%
\usepackage{amssymb}
\usepackage{amsfonts}
\usepackage{amsmath}
\usepackage{graphicx}%

\providecommand{\U}[1]{\protect\rule{.1in}{.1in}}
\numberwithin{equation}{section}
\numberwithin{equation}{section}
\numberwithin{equation}{section}
\def\e{\varepsilon}

\newtheorem{theorem}{Theorem}[section]

\newtheorem{definition}[theorem]{Definition}

\newtheorem{lemma}[theorem]{Lemma}

\newtheorem{proposition}[theorem]{Proposition}

\newtheorem{remark}[theorem]{Remark}

\newenvironment{proof}[1][Proof]{\noindent\textbf{#1.} }{\ \rule{0.5em}{0.5em}}

\begin{document}

\title{A relaxation result in $BV\times L^{p}$ for integral functionals depending on chemical composition and elastic strain.}
\author{\textsc{Gra\c{c}a Carita}\thanks{CIMA-UE, Departamento de Matem\'{a}tica,
Universidade de \'{E}vora, Rua Rom\~{a}o Ramalho, 59 7000 671 \'{E}vora,
Portugal e-mail: gcarita@uevora.pt}, 
\textsc{ Elvira Zappale}\thanks{D.I.In., Universita' degli
Studi di Salerno, Via Giovanni Paolo II 132, 84084 Fisciano (SA) Italy
e-mail:ezappale@unisa.it}}
\maketitle

\begin{abstract}
An integral representation result is obtained for the relaxation of a class of energy functionals depending on two vector fields with different behaviors which appear in the context of thermochemical equilibria and are related to image decomposition models and directors theory in nonlinear elasticity.



\end{abstract}

\section{Introduction}

In this paper we consider energies depending on two vector fields with
different behaviors: $u\in W^{1,1}\left(  \Omega;\mathbb{R}^{n}\right)  $ and
$v\in L^{p}\left(  \Omega;\mathbb{R}^{m}\right)  ,~\Omega$ being a bounded
open subset of $\mathbb{R}^{N}.$
 
Let $1<p\leq\infty$ and for every $\left(  u,v\right)  \in W^{1,1}(
\Omega;\mathbb{R}^{n})  \times L^{p}(\Omega;\mathbb{R}%
^{m})  $ define the functional
\begin{equation}\label{functional}
J\left(  u,v\right)  :=\int_{\Omega}f\left(v  ,\nabla u  \right)  dx
\end{equation}
where $f:\mathbb{R}^{m}\times\mathbb{R}%
^{n\times N}\rightarrow\left[  0,\infty\right)  $ is a continuous function.

Minimization of energies depending on two independent vector fields have
been introduced to model several phenomena. For instance the case of
thermochemical equilibria among multiphase multicomponent solids and Cosserat theories in the context of elasticity: we refer
to \cite{LDR, FKP2} and the references therein for a detailed explanation about this kind of applications.

In the Sobolev setting, after the pioneer works \cite{FKP2, LDR}, relaxation with a Carath\'{e}odory
density $f\equiv f\left( x,u,\nabla u,v\right),$ and homogenization for density of the type $f\left(\frac{x}{\varepsilon}, \nabla u, v\right)$ have been considered in \cite{CRZ2} and \cite{CRZ1}, respectively.

In the present paper we are interested in studying the lower semicontinuity and relaxation of $%
\left( \ref{functional}\right) $ with respect to the $L^{1}-$strong$\times
L^{p}-$weak convergence.
Clearly, bounded sequences $\left\{ u_{h}\right\} \subset W^{1,1}(
\Omega ;\mathbb{R}^{n}) $ may converge in $L^{1}$, up to a
subsequence, to a $BV$ function. 

In the $BV$-setting this question has been already addressed in \cite{FKP1}, only when the density $f$ is convex-quasiconvex (see \eqref{cross-qcx}) and the vector field $v \in L^\infty(\Omega;\mathbb{R}^{m})$. 

Here we allow $v$ to be in $L^p(\Omega;\mathbb{R}^{m}), p >1$ and $f$ is  not necessarily convex-quasiconvex. We provide an argument alternative to the one in \cite{FKP1}, devoted to clarify some points in the lower semicontinuity result therein.

We also emphasize that under specific restrictions on the density $f$, i.e. $f(x,u,v,\nabla u)\equiv W(x,u,\nabla u)+ \varphi (x,u,v)$, such analysis was considered already in \cite{RZCh} in order to describe image decomposition models. In \cite{RZ} a general $f$ was taken into account when the target $u$ is in $W^{1,1}(\Omega;\mathbb{R}^{n})$. 

In this manuscript we consider $f \equiv f(v, \nabla u)$ and $u \in BV(\Omega;\mathbb{R}^{n})$.

We study separately the cases $1<p<\infty$ and $p=\infty.$ To this end,
we introduce for $1<p<\infty$ the functional%
\begin{equation}
\overline{J}_{p}\left(  u,v\right)  :=\inf\left\{  \underset{h\rightarrow
\infty}{\lim\inf}J\left(  u_h,v_h\right)  :u_h\in W^{1,1}\left(
\Omega;\mathbb{R}^{n}\right)  ,~v_h\in L^{p}\left(  \Omega;\mathbb{R}%
^{m}\right)  ,~u_h\rightarrow u\text{ in }L^{1},~v_h\rightharpoonup
v\text{ in }L^{p}\right\}  , \label{relaxedp}%
\end{equation}
for any pair $(u,v)  \in BV(\Omega;\mathbb{R}^{n})
\times L^{p}(\Omega;\mathbb{R}^{m})  $ and, for $p=\infty$ the
functional%
\begin{equation}
\overline{J}_{\infty}(u,v)  :=\inf\left\{  \underset{h\rightarrow
\infty}{\lim\inf}J\left(  u_h,v_h\right)  :u_h\in W^{1,1}\left(
\Omega;\mathbb{R}^{n}\right)  ,~v_h\in L^{\infty}\left(  \Omega
;\mathbb{R}^{m}\right)  ,~u_h\rightarrow u\text{ in }L^{1},~v_h%
\overset{\ast}{\rightharpoonup}v\text{ in }L^{\infty}\right\}  ,
\label{relaxedinfty}%
\end{equation}
for any pair $\left(  u,v\right)  \in BV\left(  \Omega;\mathbb{R}^{n}\right)
\times L^{\infty}\left(  \Omega;\mathbb{R}^{m}\right)  .$

 Since bounded sequences $\{u_h\}$ in $W^{1,1}(\Omega;\mathbb R^n)$ converge in $L^1$ to a $BV$ function $u$ and bounded sequences $\{v_h\}$ in $L^p(\Omega;\mathbb R^m)$ if $1<p < \infty$, (in $L^\infty(\Omega;\mathbb R^m)$ if $p=\infty$) weakly converge to a function  $ v \in L^p(\Omega;\mathbb R^m)$, (weakly $\ast$ in $L^\infty$), the relaxed functionals  $\overline{J}_{p}$ and $\overline{J}_{\infty}$ will be composed by an absolutely continuous part and a singular one with respect to the Lebesgue measure (see \eqref{summeas}). 
 On the other hand, as already emphasized in \cite{FKP1}, it is crucial to observe that $v$, regarded as a measure, is absolutely continuous with respect to the Lebesgue one, besides it is not defined on the singular sets of $u$. Namely in those sets where the singular part with respect the Lebesgue measure of the distributional gradient of $u$, $D^su$, is concentrated. Thus specific features of the density $f$ will come into play to ensure a proper integral representation.

The integral representation of (\ref{relaxedp}) will be achieved  in Theorem \ref{MainResultp} under the following hypotheses:
\begin{itemize}

\item[$(H_1)_{p}$] There exists $C>0$ such that
\[
\frac{1}{C}\left(  \left\vert b\right\vert ^{p}+\left\vert \xi\right\vert
\right)  -C\leq f\left( b,\xi\right)  \leq C\left(  1+\left\vert
b\right\vert ^{p}+\left\vert \xi\right\vert \right)  ,
\]
for $(b,\xi)  \in \mathbb{R}^{m}\times\mathbb{R}^{n\times N}$.
\item[$(H_2)_p$] There exists $C'>0, L>0,0<\tau\leq1$ such
that
\[
t>0,~\xi\in\mathbb{R}^{n\times N},\text{ with }t\left\vert \xi\right\vert
>L\Longrightarrow\left\vert \frac{f\left(b,t\xi\right)  }{t}-f^{\infty
}\left(  b,\xi\right)  \right\vert \leq C^{\prime}\left(  
\frac{\left\vert b\right\vert ^{p}+1}{t}  +\frac{\left\vert \xi\right\vert
^{1-\tau}}{t^{\tau}}\right),
\]

where $f^\infty$ is the recession function of $f$ defined for every $ b\in  \mathbb R^m$ as 

\begin{equation}\label{recession}
\displaystyle{f^\infty(b, \xi):= \limsup_{t \to \infty}\frac{f(b, t \xi)}{t}}.
\end{equation} 
\end{itemize}

In order to characterize the functional $\overline{J}_{\infty}$ introduced in
$\left(  \ref{relaxedinfty}\right)  $ we will replace assumptions $(H_1)_p$ and $(H_2)_p$ by the following ones:
\begin{itemize}
\item[$(H_1)  _{\infty}$] Given $M>0,$ there exists $C_{M}>0$ 
such that, if $\left\vert v\right\vert \leq M$
then%
\[
\frac{1}{C_{M}}\left\vert \xi\right\vert -C_{M}\leq f\left(b,\xi\right)
\leq C_{M}\left(  1+\left\vert \xi\right\vert \right)  ,
\]
for every $\xi  \in \mathbb{R}^{n\times N}.$

\item[$(H_2)_{\infty}$] Given $M>0$, there exist $C_{M}^{\prime
}>0,~L>0,~0<\tau\leq1$ such that
\[
\left\vert b\right\vert \leq M,~t>0,~\xi\in\mathbb{R}^{n\times N},~\text{with
}t\left\vert \xi\right\vert >L\Longrightarrow\left\vert \frac{f\left(
b,t\xi\right)  }{t}-f^{\infty}\left( b,\xi\right)  \right\vert \leq
C_{M}^{^{\prime}}\frac{\left\vert \xi\right\vert ^{1-\tau}}{t^{\tau}}.
\]

\end{itemize}

Section \ref{notpre} is devoted to notations, preliminaries about measure theory and some properties of the energy densities. In particular, we stress that a series of results is presented in order to show all the properties and relations among the relaxed energy densities involved in the integral representation and that can be of further use for the interested readers since they often appear in the integral representation context. Section \ref{Main} contains the arguments necessary to prove the main results stated below.

\begin{theorem}\label{MainResultp}
Let $J$ be given by (\ref{functional}), with $f$ satisfying $(H_1)_p$ and $(H_2)_p$ and let ${\overline J}_p$ be given by \eqref{relaxedp} then 
\begin{equation}\nonumber
\overline{J}_p(u,v)= \int_\Omega   {\cal CQ}f\left(v, \nabla u\right) dx
+ \int_\Omega ({\cal CQ}f)^\infty\left(0,\frac{dD^su}{d|D^s u|}\right) d|D^s u|,
\end{equation}
for every $(u,v)\in BV(\Omega;\mathbb R^n)\times L^p(\Omega;\mathbb R^m)$.
\end{theorem}
We denote by  ${\cal CQ}f$ the convex-quasiconvex envelope of $f$ in \eqref{CQ envelope} and $({\cal CQ}f)^\infty$ represents the recession function of ${\cal CQ}f$, defined according to \eqref{recession}, which coincides, under suitable assumptions, (see assumptions \eqref{H1_1inftyCQf}, \eqref{H2_inftyCQf}, Proposition \ref{CQf=CQf} and Remark \ref{remCQfinfty}),  with the convex-quasiconvex envelope of $f^\infty$, ${\cal CQ}(f^\infty)$, and this allows us to remove the parenthesis.

For the case $p=\infty$ we have the following.
\begin{theorem}\label{MainResultinfty}
Let $J$ be given by (\ref{functional}), with $f$ satisfying $(H_1)_{\infty}$ and $(H_2)_{\infty}$ and let ${\overline J}_\infty$ be given by \eqref{relaxedinfty} then 
\begin{equation}\nonumber
\overline{J}_\infty(u,v)= \int_\Omega {\cal CQ}f\left(v,\nabla u\right) dx
+ \int_\Omega  ({\cal CQ}f)^\infty\left(0,\frac{dD^su}{d|D^s u|}\right) d|D^s u|,
\end{equation}
for every $(u,v)\in BV(\Omega;\mathbb R^n)\times L^{\infty}(\Omega;\mathbb R^m)$.
\end{theorem}

For the case $1<p<\infty$, the proof of the lower bound is presented in Theorem \ref{maintheorem}  while the upper bound is in Theorem \ref{uppergeneralp}, both under the extra hypothesis 
\begin{itemize}
\item [$(H_0)$]$f$ is convex-quasiconvex.
\end{itemize}  
The case $p=\infty$ is discussed in subsection \ref{relaxation}.
Furthermore, we observe that Proposition \ref{relaxedp=relaxedp} in subsection \ref{qcx-remove} is devoted to remove the convexity-quasiconvexity assumption on $f$.

\section{Notations preliminaries and properties of the energy densities}\label{notpre}
In this section, we start by establishing notations, recalling some preliminary results on measure theory that will be useful through the paper and finally we recall the space of functions of bounded variation. 
\noindent Then we deduce the main properties of convex-quasiconvex functions, recession functions and related envelopes. 

\noindent If $\nu \in \mathbb S^{N-1}$ and $\{\nu,\nu_2,\dots,\nu_N\}$ is an orthonormal basis of $\mathbb R^N$, $Q_\nu$ denotes the unit
cube centered at the origin with its faces either parallel or orthogonal to $\nu,\nu_2,\dots,\nu_N$. If $x \in \mathbb R^N$ and $\rho>0$, we set $Q(x,\rho):=x+\rho \, Q$ and $Q_\nu(x,\rho):=x+\rho \, Q_\nu$, $Q$ is the  cube $\left(-\frac{1}{2}, \frac{1}{2}\right)^N$. 

Let $\Omega$ be a generic open subset of $\mathbb R^N$, we denote by
${\cal M}(\Omega)$ the space of all signed Radon measures in $\Omega$ with bounded
total variation. By the Riesz Representation Theorem, ${\cal M}(\Omega)$ can
be identified to the dual of the separable space $C_0(\Omega)$ of
continuous functions on $\Omega$ vanishing on the boundary $\partial
\Omega$. The $N$-dimensional Lebesgue measure in $\mathbb R^N$ is designated
as ${\cal L}^N$.
 
If $\mu \in {\cal M}(\Omega)$ and $\lambda \in {\cal M}(\Omega)$ is a
nonnegative Radon measure, we denote by $\frac{d\mu}{d\lambda}$ the
Radon-Nikod\'ym derivative of $\mu$ with respect to $\lambda$. By a
generalization of the Besicovich Differentiation Theorem (see
\cite[Proposition 2.2]{ADM}), it can be proved that there exists a
Borel set $E \subset \Omega$ such that $\lambda(E)=0$ and
$$\frac{d\mu}{d\lambda}(x)=\lim_{\rho \to 0^+} \frac{\mu(x+\rho \, C)}{\lambda(x+\rho \, C)}$$
for all $x \in {\rm Supp }\, \lambda \setminus E$ and any open bounded convex
set $C$ containing the origin.

\noindent We recall that the exceptional set E above does not depend on C. An immediate corollary is the generalization of Lebesgue-Besicovitch Differentiation
Theorem given below.
\begin{theorem}\label{thm2.8FM2}
If $\mu$ is a nonnegative Radon measure and if $f \in L^1_{\rm loc}(\mathbb R^N,\mu)$ then
$$
\lim_{\e \to 0^+} \frac{1}{\mu(x+ \e C)}\int_{x+ \e C} | f(y) - f ( x ) | d\mu(y) =0
$$
for $\mu$- a.e. $ x\in \mathbb R^N$ and for every, bounded, convex, open set $C$ containing the
origin.
\end{theorem}

\begin{definition}
A function $u\in L^{1}(\Omega;{\mathbb{R}}^{n})$ is said to be of
\emph{bounded variation}, and we write $u\in BV(\Omega;{\mathbb{R}}^{n})$, if
all its first distributional derivatives, $D_{j}u_{i}$, belong to $\mathcal{M}%
(\Omega)$ for $1\leq i\leq n$ and $1\leq j\leq N$.
\end{definition}

The matrix-valued measure whose entries are $D_{j}u_{i}$ is denoted by $Du$
and $|Du|$ stands for its total variation.
We observe that if $u\in BV(\Omega;\mathbb{R}^n)$ then $u\mapsto|Du|(\Omega)$ is lower
semicontinuous in $BV(\Omega;\mathbb{R}^n)$ with respect to the
$L_{\mathrm{loc}}^{1}(\Omega;\mathbb{R}^n)$ topology.

By the Lebesgue Decomposition Theorem we can split $Du$
into the sum of two mutually singular measures $D^a u$ and $D^s u$,
where $D^a u$ is the absolutely continuous part and $D^s u$ is the singular part of $Du$ with respect to the Lebesgue measure ${\cal L}^N$. By $\nabla u$ we denote the Radon-Nikod\'ym derivative of $D^au$ with respect to the Lebesgue measure so that we can write $$Du=\nabla u {\cal L}^N + D^s u.$$

\begin{proposition}
\label{thm2.3BBBF} If $u\in BV\left(  \Omega;\mathbb{R}^{n}\right)  $ then
for $\mathcal{L}^{N}-$a.e. $x_{0}\in\Omega$%
\begin{equation}
\lim_{\varepsilon\rightarrow0^{+}}\frac{1}{\varepsilon}\left\{  \frac
{1}{\mathcal{\varepsilon}^{N}}\int_{Q\left(  x_{0},\varepsilon\right)
}\left\vert u(x)-u(x_{0})-\nabla u\left(  x_{0}\right)  \cdot(x-x_{0}%
)\right\vert ^{\frac{N}{N-1}}dx\right\}  ^{\frac{N-1}{N}}=0.
\label{approximate differentiability}%
\end{equation}

\end{proposition}

For more details regarding functions of bounded variation we refer to \cite{AFP}.

\subsection{Convex-quasiconvex functions}\label{remove}

We start by recalling the notion of convex-quasiconvex function, presented in
\cite{FKP1} (see also \cite{FKP2} and \cite{LDR}).

\begin{definition}
A Borel measurable function $f:\mathbb{R}^{m}\times\mathbb{R}^{n\times
N}\rightarrow\mathbb{R}$ is said to be convex-quasiconvex if, for every $(b,\xi)\in\mathbb{R}^{m}\times\mathbb{R}^{n\times N}$, there exists a
bounded open set $D$ of $\mathbb{R}^{N}$ such that
\begin{equation}
f(b,\xi)\leq\frac{1}{|D|}\int_{D}f(b+\eta(x),\xi+\nabla\varphi(x))\,dx,
\label{cross-qcx}%
\end{equation}
for every
$\eta\in L^{\infty}(D;\mathbb{R}^{m})$, with ${\int_{D}\eta(x)\,dx=0}$, and
for every $\varphi\in W_{0}^{1,\infty}(D;\mathbb{R}^{n})$.
\end{definition}

\begin{remark}

\begin{enumerate}
\item[i)] It can be easily seen that, if $f$ is convex-quasiconvex then
condition $\left(  \ref{cross-qcx}\right)  $ is true for any bounded open set
$D\subset\mathbb{R}^{N}.$

\item[ii)] A convex-quasiconvex function is separately convex.

\item[iii)] The growth condition from above in $(H_1)_p$, ii) and \cite[Proposition 2.11]{CRZ1}, entail that there exists $\gamma >0$ such that
\begin{equation}
\left\vert f\left(  b,\xi\right)  -f\left(  b^{\prime},\xi^{\prime
}\right)  \right\vert \leq\gamma\left(  \left\vert \xi-\xi^{\prime}\right\vert
+\left(  1+\left\vert b\right\vert ^{p-1}+\left\vert b^{\prime}\right\vert
^{p-1}+\left\vert \xi\right\vert ^{\frac{1}{p^{\prime}}}+\left\vert
\xi^{\prime}\right\vert ^{\frac{1}{p^{\prime}}}\right)  \left\vert
b-b^{\prime}\right\vert \right)  \label{p-lipschitzcontinuity}%
\end{equation}
for every $b,~b^{\prime}\in\mathbb{R}^{m},$ $\xi,~\xi^{\prime}\in
\mathbb{R}^{n\times N},$
where $p>1$ and $p'$ its conjugate exponent.

\item[iv)] In case of growth
conditions expressed by $(H_1)_{\infty}$ (see \cite[Proposition 4]{RZ}), ii) entails that,
given $M>0$ there exists a constant
$\beta\left(  M,n,m,N\right)  $ such that%
\begin{equation}
\left\vert f\left( b,\xi\right)  -f\left(  b^{\prime},\xi^{\prime
}\right)  \right\vert \leq\beta\left(  1+\left\vert \xi\right\vert +\left\vert
\xi^{\prime}\right\vert \right)  \left\vert b-b^{\prime}\right\vert
+\beta\left\vert \xi-\xi^{\prime}\right\vert \label{infty-lipschitzcontinuity}%
\end{equation}
for every $b,~b^{\prime}\in\mathbb{R}^{m},$ such that $\left\vert b\right\vert
\leq M$ and $\left\vert b^{\prime}\right\vert \leq M,$ for every $\xi
,~\xi^{\prime}\in\mathbb{R}^{n\times N}.$

\end{enumerate}
\end{remark}

We introduce the notion of convex-quasiconvex envelope of a function, which is crucial to deal with the relaxation procedure.

\begin{definition}\label{CQenv}
Let $f:\mathbb{R}^{m}\times\mathbb{R}^{n\times N}\rightarrow\mathbb{R}$ be a
Borel measurable function bounded from below. The convex-quasiconvex envelope
is the largest convex-quasiconvex function below $f,$ i.e.,%
\[
{\cal CQ}f\left(  b,\xi\right)  :=\sup\left\{  g\left(  b,\xi\right)  :g\leq
f,~g\text{ convex-quasiconvex}\right\}  .
\]

\end{definition}
 By Theorem 4.16 in \cite{LDR}, the convex-quasiconvex envelope coincides with the so called convex-quasiconvexification
\begin{equation}
\begin{array}{ll}
{\cal CQ}f(b,\xi)=\inf \left\{\displaystyle{\frac{1}{|D|}\int_{D}f(b+\eta(x),\xi+\nabla
\varphi(x))\,dx:} \right.\displaystyle{\eta\in L^{\infty}\left(  D;\mathbb{R}^{m}\right)}  ,&\displaystyle{\int_{D}\eta(x)  dx=0,} \\
&\left. \varphi\in W_{0}^{1,\infty}(
D;\mathbb{R}^{n})  \right\}  .
\end{array}\label{CQ envelope}%
\end{equation}
As for convexity-quasiconvexity, condition \eqref{CQ envelope}
can be stated for any bounded open set $D\subset\mathbb{R}^{N}.$ It can also be showed that if $f$ satisfies a growth condition of type $(H_1)_p$ then in \eqref{cross-qcx} and \eqref{CQ envelope} the spaces $L^{\infty}$ and $W_{0}^{1,\infty}$ can be replaced by $L^{p}$ and $W_{0}^{1,1},$ respectively.

The following proposition, that will be exploited in the sequel, can be found in \cite[Proposition 5]{RZ}. The proof is omitted since it is very similar to \cite[Proposition 2.1]{RZCh}.

\begin{proposition}\label{H12pCQf}
Let $f:\mathbb{R}^{m}\times\mathbb{R}^{n\times
N}\rightarrow\left[  0,\infty\right)  $ be a continuous function satisfying
$\left(  H_{1}\right)  _{p}$. Then ${\cal CQ}f$ is
continuous and satisfies $\left(  H_{1}\right)  _{p}.$ Consequently, ${\cal CQ}f$ satisfies \eqref{p-lipschitzcontinuity}.
\end{proposition}

In order to deal with $v \in L^\infty(\Omega;\mathbb{R}^{m})$ and to compare with the result in $BV \times L^p$, $1<p <\infty$, one can consider a different setting of assumptions on the energy density $f$.
 
\noindent Namely, following \cite[Proposition 6 and Remark 7]{RZ},
if $\alpha:[0, \infty) \to [0,\infty)$ is a convex and increasing function, such that $\alpha(0)=0$ and if
$f: \mathbb R^m \times \mathbb R^{n \times N} \to \mathbb [0, \infty)$
is a continuous function satisfying
\begin{equation}\label{H1_1inftyCQf}
\frac{1}{C}(\alpha(|b|)+ |\xi|)-C \leq f(b,\xi) \leq C(1+ \alpha(|b|)+ |\xi|)
\end{equation}
for every $(b,\xi)\in \mathbb R^m \times \mathbb R^{n \times N}$,
then ${\cal CQ}f$ satisfies a condition analogous to \eqref{H1_1inftyCQf}. Moreover, ${\cal CQ}f$ is a continuous function.

\noindent Analogously, one can assume that $f$ satisfies the following variant of $(H_2)_\infty$: there exist $c'>0, L>0,0<\tau\leq1$ such
that
\begin{equation}
t>0,~\xi\in\mathbb{R}^{n\times N},\text{ with }t\left\vert \xi\right\vert
>L\Longrightarrow\left\vert \frac{f\left(b,t\xi\right)  }{t}-f^{\infty
}\left(  b,\xi\right)  \right\vert \leq c^{\prime}\left(  
\frac{\alpha(|b|)+1}{t}  +\frac{\left\vert \xi\right\vert
^{1-\tau}}{t^{\tau}}\right).
\label{H2_inftyCQf}
\end{equation}

\noindent We observe that, if from one hand \eqref{H1_1inftyCQf} and \eqref{H2_inftyCQf} generalize $(H_1)_p$ and $(H_2)_p$ respectively, from the other hand they can be regarded also as a stronger version of $(H_1)_{\infty}$ and $(H_2)_{\infty},$ respectively.

\subsection{The recession function}

Let $f: \mathbb R^m \times \mathbb R^{n \times N}\to [0, \infty[$, and let $f^\infty: \mathbb R^m \times \mathbb R^{n \times N}\to [0, \infty[$  be its recession function, defined in \eqref{recession}.

\noindent The following properties are an easy consequence of the definition of
recession function and conditions $(H_0)$, $(H_1)_p$ and $(H_2)_p$, when $1<p<\infty$. 

\begin{proposition}
\label{proprecession}

Provided $f$ satisfies $(H_0)$, $(H_1) _p$ and $(H_2)_{p}$, then 
\begin{enumerate}
\item $f^\infty$ is convex-quasiconvex;

\item there exists $C >0$ such that%
\begin{equation}\label{finftygrowth}
\frac{1}{C}\left\vert \xi\right\vert \leq f^{\infty}\left(b,\xi\right)
\leq C\left\vert \xi\right\vert;
\end{equation}
\item $f^{\infty}(b,\xi)$ is constant with respect to $b$ for every $\xi\in \mathbb R^{n \times N}$;

\item $f^\infty$ is continuous.

\end{enumerate}
\end{proposition}

\begin{remark}
\label{hypmin}
We emphasize that not all the assumptions $(H_1)_p$ and $(H_2)_p$ 
in Proposition \ref{proprecession} are necessary to prove items above. In particular, one has that:
\begin{itemize}
\item[i)] The proof of $2.$  uses only the fact that $f$ satisfies $(H_1)_p$.
\item[ii)] To prove $3.$ it is necessary to require that $f$ satisfies only $(H_0)$ and $(H_1)_p$. In fact, under the assumptions that $f$ satisfies \eqref{p-lipschitzcontinuity} one can avoid to require $(H_0)$.
 
\end{itemize}
\end{remark}
\begin{proof}

\begin{enumerate}
\item The convexity-quasiconvexity of $f^\infty$ can be proven exactly as in \cite[Lemma 2.1]{FKP1}.

\item By definition $\left(  \ref{recession}\right)  $ we may find a
subsequence $\left\{  t_{k}\right\}  $ such that%
\[
f^{\infty}\left( b,\xi\right)  =\lim_{t_k\rightarrow\infty}\frac{f\left(
b,t_{k}\xi\right)  }{t_{k}}.
\]
By $\left(  H_{1}\right)  _{p}$ one has%
\[
f^{\infty}\left(  b,\xi\right)  \leq\lim_{t_k\rightarrow\infty}%
\frac{C\left(  1+\left\vert b\right\vert ^{p}+\left\vert t_{k}\xi\right\vert
\right)  }{t_{k}}=C\left\vert \xi\right\vert
\hbox{ and} 
f^{\infty}\left(  b,\xi\right)  \geq\lim_{t_k\rightarrow \infty}\frac
{\frac{1}{C}\left(  \left\vert b\right\vert ^{p}+\left\vert t_{k}%
\xi\right\vert \right)  -C}{t_{k}}\geq\frac{1}{C}\left\vert \xi\right\vert .
\]
Hence $\left(  H_{1}\right)  _{p}$ holds for $f^{\infty}.$
\item We start by observing that \eqref{finftygrowth} and 1. guarantee that $f^\infty$ satisfies \eqref{p-lipschitzcontinuity}. Let $\xi\in \mathbb R^{n \times N}$, and let $b, b' \in \mathbb R^m$, up to a subsequence, by \eqref{recession} and \eqref{p-lipschitzcontinuity} it results that, 
$$
\begin{array}{ll}
\displaystyle{f^\infty(b,\xi)- f^\infty( b', \xi)\leq \lim_{t_k \to \infty}\frac{f(b, t_k \xi)- f( b', t_k \xi)}{t_k}}
\displaystyle{\leq \lim_{t_k \to \infty}\frac{\gamma (1+ |b|^{p-1}+ |b'|^{p-1}+ |t_k \xi|^{\frac{1}{p'}})|b-b'|}{t_k}=0.}
\end{array}
$$
\end{enumerate}
By interchanging the role of $b$ and $b'$, it follows that $f^{\infty}( \cdot, \xi)$ is constant and this concludes the proof.
\end{proof}

\begin{remark}
\label{finftyFKP1}
Under assumptions $(H_0), (H_1)_{\infty}$ and $(H_2)_{\infty}$, $f^\infty$ satisfies properties analogous to those at the beginning of subsection ???
In particular in \cite[Lemma 2.1 and Lemma 2.2]{FKP1} it has been proved that
\begin{itemize}
\item[i)] $f^\infty$ is convex-quasiconvex;
\item[ii)] $\frac{1}{C_M} |\xi|\leq f^\infty(b,\xi) \leq C_M |\xi|$, for every $b$, with $|b|\leq M$;
\item[iii)] If ${\rm rank} \xi \leq 1$, then $f^\infty(b,\xi)$ is constant with respect to $b$.
\end{itemize}
\end{remark}

\begin{remark}
\label{propCQfinfty}
We observe that, if $f:\mathbb{R}^{m}\times\mathbb{R}^{n\times
N}\rightarrow\left[  0,\infty\right)  $ is a continuous function satisfying 
$(H_1)_p$ and $(H_2)_p,$ then the function $({\cal CQ}f)^\infty: \mathbb R^m \times \mathbb R^{n\times N} \to [0,\infty[$, obtained first taking the convex-quasiconvexification in \eqref{CQ envelope} of $f$ and then its recession through formula \eqref{recession} applied to ${\cal CQ}f$, satisfies the following properties:
\begin{enumerate}
\item $({\cal CQ}f)^\infty$ is convex-quasiconvex;
\item there exists $C>0$ such that $\frac{1}{C}|\xi|\leq ({\cal CQ}f)^\infty(b,\xi)\leq C|\xi|$, for every $(b,\xi) \in \mathbb R^m \times \mathbb R^{n \times N}$;
\item for every $\xi\in \mathbb R^{n \times N}$, $({\cal CQ}f)^\infty(\cdot,\xi)$ is constant, i.e. $({\cal CQ}f)^\infty$ is independent on $v$;
\item $({\cal CQ}f)^\infty$ is Lipschitz continuous in $\xi$.
\end{enumerate}

Under the same set of assumptions on $f$, one can prove that the convex-quasiconvexification of $f^\infty$, ${\cal CQ}(f^\infty)$, satisfies the following conditions:
\begin{enumerate}
\item[5.] ${\cal CQ}(f^\infty)$  is convex-quasiconvex;
\item[6.] there exists $C>0$ such that $\frac{1}{C}|\xi|\leq {\cal CQ}(f^\infty)(b,\xi)\leq C|\xi|$, for every $(b,\xi)\in\mathbb R^m \times \mathbb R^{n \times N}$;
\item[7.] for every $\xi\in \mathbb R^{n \times N}$, and assuming that $f$ satisfies \eqref{p-lipschitzcontinuity}, ${\cal CQ}(f^\infty)(\cdot,\xi)$ is constant, i.e. ${\cal CQ}(f^\infty)$ is independent on $b$;
\item[8.] ${\cal CQ}(f^\infty)$ is Lipschitz continuous in $\xi$.
\end{enumerate}
\end{remark}
The above properties are immediate consequences of Propositions \ref{H12pCQf}, \ref{proprecession} and \eqref{p-lipschitzcontinuity}. In particular $8.$ follows from $3.$ of Proposition \ref{proprecession}, without requiring $(H_2)_p$. 

On the other hand, Proposition \ref{CQf=CQf} below entails that ${\cal CQ}(f)^\infty$ is independent on $b$, without requiring that $f$ is Lipschitz continuous, but replacing this assumption with $(H_2)_p$.

\noindent We also observe that $({\cal CQ}f)^\infty$ and ${\cal CQ}(f^\infty)$ are only quasiconvex functions, since they are independent of $b$. In particular, in our setting, these functions coincide as it is stated below.

\begin{proposition}\label{CQf=CQf}
Let $f:\mathbb{R}^{m}\times\mathbb{R}^{n\times
N}\rightarrow\left[  0,\infty\right)  $ be a continuous function satisfying 
$(H_1)_p$ and $(H_2)_p.$ 

Then
\[
{\cal CQ}\left(  f^{\infty}\right)  \left(  b,\xi\right)  =\left(  {\cal CQ}f\right)
^{\infty}\left( b,\xi\right)  ~\ \ \ \ \ \ \text{for every }\left(
b,\xi\right)  \in\mathbb{R}^{m}%
\times\mathbb{R}^{n\times N}.
\]
 
\end{proposition}

\begin{proof}
The proof will be achieved by double inequality.

For every $(b,\xi)\in \mathbb R^m\times \mathbb R^{n\times N}$ the inequality
\begin{equation}
\label{firstineqCQfinfty}
\displaystyle{({\cal CQ}f)^\infty(b,\xi)\leq {\cal CQ}(f^\infty)(b,\xi)}
\end{equation}
 follows by Definition \ref{CQenv}, and the fact that
${\cal CQ}f(b,\xi) \leq f(b,\xi) $.
In fact, \eqref{recession} entails that the same inequality holds when, passing to $(\cdot)^\infty$. Finally, $1.$ in Proposition \ref{proprecession}, guarantees \eqref{firstineqCQfinfty}.

In order to prove the opposite inequality, fix $(b,\xi)\in\mathbb R^m \times \mathbb R^{n \times N}$ and, for every $t >1$, take $\eta_t \in L^\infty(Q;\mathbb R^m)$, with $0$ average, and $\varphi_t \in W^{1,\infty}_0(Q;\mathbb R^n)$ such that
\begin{equation}\label{3.7BZZ}
\int_Q f( b+\eta_t, t \xi + \nabla \varphi_t(y))\,dy \leq {\cal CQ}f(b,t \xi )+ 1.
\end{equation}
By $(H_1)_p$ and Proposition \ref{H12pCQf}, we have 
that $\|b+\eta_t\|_{L^p(Q)}$, $\left\|\nabla(\frac{1}{t}\varphi_t)\right\|_{L^1(Q)}\leq C $ for a constant independent on $t$. Defining $\psi_t:=\frac{1}{t}\varphi_t$, one has $\psi_t \in W^{1,\infty}_0(Q;\mathbb R^n)$ and thus
$${\cal CQ}(f^\infty)(b,\xi)\leq \int_Q f^\infty(b+\eta_t, \xi + \nabla \psi_t(y))\,dy.$$

Let $L$ be the constant appearing in condition $(H_2)_p$. We split the cube $Q$ in the set $\{y\in Q:\ t|\xi+\nabla \psi_t(y)|\le L\}$ and its complement in $Q$. Then we apply condition $(H_2)_p$ and \eqref{finftygrowth}  to get
$$\begin{array}{rcl}{\cal CQ}(f^\infty)(b,\xi) & \leq & \displaystyle \int_Q \left(C\frac{1+|b+\eta_t|^p}{t}+ C\frac{|\xi+\nabla \psi_t|^{1-\tau}}{t^\tau}+\frac{f(b+ \eta_t,t\xi+\nabla \varphi_t)}{t}+C\frac{L}{t}\right)\,dy.
\end{array}$$

Applying H\"{o}lder inequality and (\ref{3.7BZZ}), we get
$$\begin{array}{rcl}{\cal CQ}(f^\infty)(b,\xi) & \leq & \displaystyle \frac{C}{t^\tau}\left(\int_Q |\xi+\nabla \psi_t|\,dy\right)^{1-\tau}+\frac{{\cal CQ}f(b, t \xi )+ 1}{t}+C\frac{L}{t}+ \frac{C'}{t},
\end{array}$$
and the desired inequality follows by definition of $({\cal CQ}f)^\infty$ and using the fact that $\nabla\psi_t$ has bounded $L^1$ norm, letting $t$ go to $\infty$.
\end{proof}

\begin{remark}
\label{remCQfinfty}
It is worth to observe that inequality
$$
\left({\cal CQ}  f^{\infty}\right)  \left(  b,\xi\right)  \leq  {\cal CQ}\left(f
^{\infty}\right)\left( b,\xi\right)  ~\ \ \ \ \ \ \text{for every }\left(
b,\xi\right)  \in\mathbb{R}^{m}%
\times\mathbb{R}^{n\times N},
$$
has been proven without requiring neither $(H_1)_p$ and $(H_2)_p$ on $f$, nor $(H_1)_\infty$ and $(H_2)_\infty$.

Furthermore, we emphasize that the proof of Proposition \ref{CQf=CQf} cannot be performed in the same way in the case $p=\infty$, with assumptions $(H_1)_p$ and $(H_2)_p$ replaced by $(H_1)_\infty$ and $(H_2)_\infty$. Indeed, an $L^\infty$ bound on $b+ \eta_t$ analogous to the one in $L^p$ cannot be obtained from $(H_1)_\infty$. On the other hand, it is possible to deduce the equality between ${\cal CQ}f^\infty$ and $({\cal CQ}f)^\infty$, when $f$ satisfies  \eqref{H1_1inftyCQf} and \eqref{H2_inftyCQf}.
\end{remark}

\subsection{Auxiliary results}\label{qcx-remove}
Here we prove that assumption $(H_0)$ on $f$ is not necessary to provide an integral representation for ${\overline J}_p$ in \eqref{relaxedp}.

Indeed, we can assume that $f: \mathbb R^m \times \mathbb R^{n\times N}\to [0, \infty[$ is a continuous function and satisfies assumptions $(H_1)_p$ and $(H_2)_p$, ($p \in (1,\infty]$). First we extend, with an abuse of notation, the functional $J$ in \eqref{functional}, to $L^1(\Omega;\mathbb R^n) \times L^p(\Omega;\mathbb R^m)$, $p \in (1,\infty]$, as
\begin{equation}
\label{functionalextended}
J (u,v):=\left\{
\begin{array}{ll}
\displaystyle{\int_\Omega f(v,\nabla u)dx }& \hbox{ if } (u,v)\in W^{1,1}(\Omega;\mathbb R^n)\times L^p(\Omega; \mathbb R^m),\\
\\
\infty &\hbox{ otherwise.}
\end{array}
\right.
\end{equation} 
Then we define, according to Definition \ref{CQenv} the convex-quasiconvex envelope of $f$, ${\cal CQ}f$, and introduce, in analogy with  
\eqref{functionalextended} and \eqref{relaxedp}, 
the functional
\begin{equation}
\nonumber
J_{{\cal CQ} f}\left(  u,v\right)  :=\left\{
\begin{array}{ll}
\displaystyle{\int_{\Omega}{\cal CQ }f\left(v,\nabla u  \right)  dx} &\hbox{ if } (u,v) \in W^{1,1}(\Omega;\mathbb R^n)\times L^p(\Omega;\mathbb R^m),\\
\\
\infty &\hbox{otherwise,}
\end{array}
\right.
\end{equation}
($p \in (1,\infty]$) and,
\begin{equation}
\nonumber
\overline{J_{{\cal CQ}f}}_{p}\left(  u,v\right)  :=\inf\left\{  \underset{h\rightarrow
\infty}{\lim\inf}J_{{\cal CQ}f}\left(  u_h,v_h\right)  :u_h\in W^{1,1}\left(
\Omega;\mathbb{R}^{n}\right)  ,~v_h\in L^{p}\left(  \Omega;\mathbb{R}%
^{m}\right)  ,~u_h\rightarrow u\text{ in }L^{1},~v_h\rightharpoonup
v\text{ in }L^{p}\right\},
\end{equation} 
for any pair $\left(  u,v\right)  \in BV\left(  \Omega;\mathbb{R}^{n}\right)
\times L^{p}\left(  \Omega;\mathbb{R}^{m}\right), p \in (1,\infty)$.
Analogously, one can consider

\begin{equation}
\nonumber
\overline{J_{{\cal CQ}f}}_{\infty}\left(  u,v\right)  :=\inf\left\{  \underset{h\rightarrow
\infty}{\lim\inf}J_{{\cal CQ}f}\left(  u_h,v_h\right)  :u_h\in W^{1,1}\left(
\Omega;\mathbb{R}^{n}\right)  ,~v_h\in L^{p}\left(  \Omega;\mathbb{R}%
^{m}\right)  ,~u_h\rightarrow u\text{ in }L^{1},~v_h\overset{\ast}{\rightharpoonup}
v\text{ in }L^{\infty}\right\},
\end{equation}
for any pair $\left(  u,v\right)  \in BV\left(  \Omega;\mathbb{R}^{n}\right)
\times L^{\infty}\left(  \Omega;\mathbb{R}^{m}\right).$

\noindent Clearly, it results that for every $\left(  u,v\right)  \in BV\left(  \Omega;\mathbb{R}^{n}\right)
\times L^p\left(  \Omega;\mathbb{R}^{m}\right)$,
\begin{equation}
\nonumber
\overline{J_{{\cal CQ}f}}_{p}\left(  u,v\right)\leq {\overline J}_p(u,v),
\end{equation}
but, as in \cite[Lemma 8 and Remark 9]{RZ}, the following proposition can be proven. 

\begin{proposition}
\label{relaxedp=relaxedp}
Let $p \in (1,\infty]$ and consider the functionals $J$ and $J_{CQf}$ and their corresponding relaxed functionals $\overline{J}_p$  and $\overline{J_{{\cal CQ}f}}_{p}$. If $f$ satisfies conditions $(H_1)_p$ and $(H_2)_p$ if $p \in (1,\infty)$, and both $f$ and ${\cal CQ}f$ satisfy $(H_1)_\infty $ and $(H_2)_\infty$ if $p=\infty$, then
\begin{equation}\nonumber
\displaystyle{\overline{J}_p(u,v)=\overline{J_{{\cal CQ}f}}_p(u,v)}
\end{equation} for every $(u,v) \in BV(\Omega;\mathbb{R}^n)\times L^p(\Omega;\mathbb R^m), p \in (1,\infty]$.
\end{proposition}

\begin{remark}
\label{obs=relaxe}
The argument has not been shown since it is already contained in \cite[Lemma 8 and Remark 9]{RZ}. In \cite{RZ} it is not required that $f$ satisfies $(H_2)_p$, $(p \in (1,\infty])$. Indeed, the coincidence between the two functionals  $\overline{J}_p$  and $\overline{J_{{\cal CQ}f}}_p$ holds independently on this assumption on $f$, but in order to remove hypothesis $(H_0)$ from the representation theorem we need to assume that ${\cal CQ}f$ inherits the same properties as $f$, which is the case as it has been observed in Proposition \ref{H12pCQf}.
It is also worth to observe that, when $p=\infty$, 

\eqref{H2_inftyCQf} is equivalent to
\begin{equation}
\nonumber
|f^\infty(b,\xi)- f(b,\xi)|\leq C(1+\alpha(|b|)+ |\xi|)
\end{equation}
for every $(b,\xi) \in \mathbb R^m \times \mathbb R^{n \times N}$, and this latter property is inherited by ${\cal CQ}f$ and ${\cal CQ}f^\infty$ as it can be easily verified arguing as in \cite[Proposition 2.3]{RZCh}. Thus Proposition \ref{relaxedp=relaxedp} holds when $p=\infty$ just requiring that $f$ satisfies \eqref{H1_1inftyCQf} and \eqref{H2_inftyCQf}.

\end{remark}

\noindent The following result can be deduced in  full analogy with \cite[Theorem 12]{RZ}, where it has been proven for $\overline{J}_\infty$.

\begin{proposition}
\label{overlineJpvariational}
Let $\Omega$ be a bounded and open  set of $\mathbb R^N$ and let $ f:\mathbb R^m \times \mathbb R^{n \times N}\to \mathbb R$ be a continuous function satisfying $(H_1)_p$ and $(H_2)_p$, $1<p\leq\infty$. Let $J$ be the functional defined in \eqref{functional}, then $\overline{J}_p$ in \eqref{relaxedp} ($1<p<\infty$), \eqref{relaxedinfty} $(p =\infty)$ is a variational functional.
\end{proposition}

By virtue of this result, it turns out that for every $(u,v)\in BV(\Omega;\mathbb R^n)\times L^p(\Omega;\mathbb R^m)$, $\overline{J}_p(u,v; \cdot)$, $(p\in (1, \infty])$ is the restriction to the open subsets in $\Omega$ of a Radon measure on $\Omega$, thus it can be decomposed as the sum of two terms
\begin{equation}\label{summeas}
\displaystyle{\overline{J}_p(u,v;\cdot)= \overline{J}_p^a(u,v; \cdot)+ \overline{J}_p^s(u,v;\cdot),}
\end{equation}
where $\overline{J}_p^a(u,v; \cdot)$ and $\overline{J}_p^s(u,v; \cdot)$ denote the absolutely continuous part and the singular part with respect to the Lebesgue measure, respectively. Next proposition deals with the scaling properties of $\overline{J}_p$.
\begin{proposition}
\label{scalingproperties}
Let $f:\mathbb R^m \times \mathbb R^{n \times N}\to \mathbb R$ be a continuous and convex-quasiconvex function, let $J$ and $\overline{J}_p$ be the functionals defined respectively by \eqref{functional} and \eqref{relaxedp} when $p \in (1,\infty]$, respectively (\eqref{relaxedinfty}, when $p=\infty)$. Then the following scaling properties are satisfied
\begin{equation}
\label{scaling}
\begin{array}{ll}
\displaystyle{\overline{J}_p(u+\eta, v;\Omega)= \overline{J}_p(u,v;\Omega) \hbox{ for every }\eta \in \mathbb R^n,}\\
\\
\displaystyle{\overline{J}_p\left(u(\cdot-x_0), v(\cdot-x_0);x_0+\Omega\right)=\overline{J}_p(u(\cdot), v(\cdot); \Omega) \hbox{ for every }x_0 \in \mathbb R^N,}
\\
\\
\displaystyle{\overline{J}_p\left(u_\varrho, v_\varrho;\frac{\Omega-x_0}{\varrho}\right)=\varrho^{-N}\overline{J}_p(u,v;\Omega),}
\end{array}
\end{equation}
where $u_\varrho(y):= \frac{u(x_0+ \varrho y)- u(x_0)}{\varrho}$ and $v_\varrho(y):= v(x_0+ \varrho y)$, for $y \in \frac{\Omega-x_0}{\varrho}.$
\end{proposition}

The following result will be exploited in the sequel. The proof is omitted since it develops along the lines of \cite[Lemma 5.50]{AFP}, the only differences being the presence of $v$ and the convexity-quasiconvexity of $f$.  
\begin{lemma}\label{Lemma5.50AFP}
Let $f:\mathbb R^m \times \mathbb R^{N \times n}\to \mathbb R$ be a continuous and convex-quasiconvex function, and let $J$ and $\overline{J}_p$ be the functionals defined respectively by \eqref{functional} and \eqref{relaxedp}. Let $\nu\in S^{N-1}$, $\eta \in S^{n-1}$ and $\psi:\mathbb R \to \mathbb R$, bounded and increasing. Denoted by $Q$ the cube $Q_\nu$, let $u \in BV(Q;\mathbb R^n)$ be representable in $Q$ as
\begin{equation}
\nonumber
\displaystyle{u(y)=\eta \psi(y \cdot \nu),}
\end{equation} 
and let $w\in BV(Q;\mathbb R^n)$ be such that ${\rm supp}(w-u)\subset \subset Q.$ Let $v \in L^p(Q;\mathbb R^m)$.
Then
$$
\displaystyle{\overline{J}_p(w,v;Q)\geq f\left(\int_Q v dy, Dw(Q)\right).}
$$
\end{lemma}

\section{Main Results}\label{Main}

This section is devoted to deduce the results stated in Theorems \ref{MainResultp} and \ref{MainResultinfty}.
We start by proving the lower bound in the case $1<p<\infty$.
For what concerns the upper bound we present, for the reader's  convenience, a self contained proof in Theorem \ref{uppergeneralp}. For the sake of completeness we observe that the upper bound, in the case $1<p<\infty$, could be deduced as a corollary from the case $p=\infty$ (see Theorem \ref{MainResultinfty}), which, in turn, under slightly different assumptions, is contained in \cite{FKP1}. 
\subsection{Lower semicontinuity in $BV\times L^{p},~1<p<\infty$}

\begin{theorem}\label{maintheorem}
Let $\Omega$ be a bounded open set of $\mathbb{R}^{N}$, let $f:\mathbb{R}^{m}\times\mathbb{R}^{n\times
N}\rightarrow\left[  0,\infty\right)  $ be a continuous function satisfying
$(H_0), \left(  H_{1}\right)  _{p} $ and $\left(  H_2\right)  _{p}$, and let $\overline{J}_p$ be the functional defined in \eqref{relaxedp}. Then
\begin{equation}\label{lowerboundLp}
\overline{J}_p(u,v; \Omega)
  \geq\int_{\Omega}f(v,\nabla
u)  dx
 +\int_{\Omega}f^\infty\left( 0,\frac{dD^s%
u}{d\left\vert D^s u\right\vert } \right)  d\left\vert
D^s u\right\vert 
\end{equation}
for any $(u,v)\in BV(\Omega;\mathbb{R}^{n})  \times L^{p}(\Omega
;\mathbb{R}^{m}). $ 
\end{theorem}

\begin{proof}
The proof will be achieved, in two steps, namely by showing that 

\begin{align}
\lim_{\varrho \to 0^+}\frac{\overline{J_p}(u,v; Q  (x_{0};\varrho))}{{\cal L}^N(Q(x_0,\varrho))}   &  \geq f(v(x_{0})  ,\nabla u(x_{0}))
,~\ \ \ \ \text{for }\mathcal{L}^{N}-\text{a.e. }x_{0}\in\Omega
,\label{lowerboundbulk}\\
\\
\lim_{\varrho \to 0^+}\frac{\overline{J_p}\left(u,v; Q(x_0,\varrho) \right)}{|Du|(Q(x_0,\varrho))}   &  \geq f^{\infty}\left(0,\frac{dD^{s}u}{d\left\vert D^{s}u\right\vert }\left(
x_{0}\right)  \right)  \text{, \ for }\left\vert D^{s}u\right\vert
-\text{a.e.~}x_{0}\in\Omega. \label{lowerboundsingular}%
\end{align}

\noindent Indeed, if $\left(  \ref{lowerboundbulk}\right) $ and $ \left(  \ref{lowerboundsingular}%
\right)  $ hold then, by virtue of \eqref{summeas}, and \cite[Theorem 2.56]{AFP}, \eqref{lowerboundLp} follows
immediately.

\noindent\textbf{Step 1.} Inequality $\left(  \ref{lowerboundbulk}\right)  $
is obtained through an argument entirely similar to \cite[Proposition 5.53]{AFP} and exploiting \cite[Theorem 1.1]{RZ}.

For ${\cal L}^N-$a.e. $x_0\in \Omega$ it results that $u$ is approximately differentiable (see \eqref{approximate differentiability}) and 
\begin{equation}
\nonumber
\lim_{\varrho \to 0^+}\frac{1}{{\cal L}^N(Q(x_0,\varrho))}\int_{Q(x_0,\varrho)}|v(x)-v(x_0)|dx = 0.
\end{equation}

\noindent Consequently, given $\varrho >0$, and defined $u_\varrho$ and $v_\varrho$ as in Proposition \ref{scalingproperties}, 

it results that $u_\varrho \to u_0$ in $L^1(\Omega;\mathbb R^n)$, where $u_0:= \nabla u(x_0)x$ and $v_\varrho \to v(x_0)$ in $L^p(\Omega;\mathbb R^m)$.
Then the scaling properties \eqref{scaling}, and the lower semicontinuity of $\overline{J}_p$ entail that
\begin{equation}
\label{scaling estimate}
\displaystyle{\liminf_{\varrho \to 0^+}\frac{\overline{J}_p(u,v; Q(x_0,\varrho))}{\varrho^N} = \liminf_{\varrho \to 0^+}\overline{J}_p(u_\varrho, v_\varrho; Q)\geq \overline{J}_p(u_0, v(x_0);Q).}
\end{equation}
Then the lower semicontinuity result proven in \cite[Theorem 11]{RZ}, when $u$ is in $W^{1,1}(\Omega;\mathbb R^n)$ and $v \in L^p(\Omega;\mathbb R^m)$, allows us to estimate the last term in \eqref{scaling estimate} as follows
$$
\displaystyle{\overline{J}_p(u_0,v(x_0);Q)\geq f(v(x_0),\nabla u(x_0)),}
$$
and that provides \eqref{lowerboundbulk}.

\noindent\textbf{Step 2.} Here we present the proof of \eqref{lowerboundsingular}. To this end we exploit techniques very similar to \cite{ADM} (see \cite[Proposition 5.53]{AFP}). Let $Du= z |Du|$ be the polar decomposition of $Du$ (see \cite[Corollary 1.29]{AFP}), for $z \in S^{N\times n-1}$, and recall that for $|D^s u|$- a.e. $x_0$, $z(x_0)$ admits the representation $\eta(x_0)\otimes \nu(x_0)$, with $\eta(x_0)\in S^{n-1}$ and $\nu(x_0) \in S^{N-1}$, (see \cite[Theorem 3.94]{AFP}). In the following, we will denote the cube $Q_\nu(x_0,1)$ by $Q$.

To achieve \eqref{lowerboundsingular} it is enough to show that 
$$
\displaystyle{\lim_{\varrho \to 0^+}\frac{{\overline J}_p(u,v; Q(x_0,\varrho))}{|Du|(Q(x_0, \varrho))} \geq f^\infty(0, z(x_0))}
$$
at any Lebesgue point $x_0$ of $z$ relative to $|Du|$ such that the limit on the left hand side exists and 
\begin{equation}
\label{Lebpointsingular}
z(x_0)=\eta(x_0)\otimes \nu(x_0),  \;\;\;\;\;\;\;\;\;\;\; \lim_{\varrho \to 0^+}\frac{|Du|(Q(x_0,\varrho))}{\varrho^N}=\infty,
\end{equation}
\begin{equation}
\label{0derivativev}
\displaystyle{0= \lim_{\varrho \to 0^+}\frac{\int_{Q(x_0,\varrho)} |v|^p dx}{|Du|(Q(x_0,\varrho))} = \lim_{\varrho \to 0^+}\frac{\int_{Q(x_0,\varrho)} |v|dx}{|Du|(Q(x_0,\varrho))}.}
\end{equation}
The above requirements are, indeed, satisfied at $|D^s u|$-a.e. $x_0 \in \Omega$, by Besicovitch's derivation theorem and Alberti's rank-one theorem (see \cite[Theorem 3.94]{AFP}).
Set $\eta\equiv \eta(x_0)$ and $\nu\equiv \nu(x_0)$, for $\varrho < N^{-\frac{1}{2}}{\rm dist}(x_0,\partial \Omega )$, define
\begin{equation}\nonumber
u_\varrho(y):=\frac{u(x_0+\varrho y)-\tilde{u}_\varrho}{\varrho}\frac{\varrho^N}{|Du|(Q(x_0,\varrho))}, \;\;\;\;\; y \in Q,
\end{equation}
where $\tilde{u}_\varrho$ is the average of $u$ in $Q(x_0,\varrho)$.
Analogously define, as in Proposition \ref{scalingproperties},
\begin{equation}
\label{vrho}
\displaystyle{v_\varrho(y):= v(x_0+ \varrho y), \;\;\; y \in Q.}
\end{equation}
Let us fix $ t \in (0,1)$. By \cite[formula (2.32)]{AFP}, there exists a sequence $\{\varrho_h\}$ converging to $0$ such that
\begin{equation}
\label{5.79AFP}
\displaystyle{\lim_{h \to \infty} \frac{|Du|(Q(x_0,t \varrho_h))}{|Du|(Q(x_0,\varrho_h))}\geq t^N.}
\end{equation}
Denote $u_{\varrho_h}$ by $u_h$, then $|D u_h|(Q)=1$ and, passing to a not relabelled subsequence, $\{u_h\}$ converges in $L^1(Q;\mathbb R^n)$ to a $BV$ function $\overline{u}$. Correspondingly, denote $v_{\varrho_h}$ by $v_h $. Then, arguing as in \cite[Proof of Proposition 5.53]{AFP} we have
\begin{equation}
\label{5.80AFP}
|D \overline{u}|(Q)\leq 1 \;\;\;\;\;\;\;\;\;\;\; \hbox{ and } \;\;\;\;\;\;\;\;\;\;\; |D \overline{u}|(\overline{Q}_t
)\geq t^N,
\end{equation} 
where $Q_t:=t Q$. It results that $\overline{u}(y)=\eta \psi(y\cdot \nu)$, for some bounded increasing function $\psi$ in $\left(-\frac{1}{2},\frac{1}{2}\right)$.  Take $\varphi \in C^1_c(Q)$ such that $\varphi=1$ on $\overline{Q}_t$ and $0 \leq \varphi \leq 1$, and let us define $w_h:=\varphi u_h + (1-\varphi)\overline{u} .$ The functions $w_h$ converge to $\overline u$ in $L^1(Q;\mathbb R^n)$ and moreover we have 
$$
\begin{array}{ll}
\displaystyle{|D(w_h- u_h)|(Q) \leq |D(u_h- {\overline u})|(Q \setminus \overline{Q}_t) + \int_Q |\nabla \varphi||u_h- {\overline u}|dy}\\
\\
\displaystyle{\leq|D u_h|(Q \setminus \overline{Q}_t) + |D {\overline u}|(Q \setminus {\overline Q_t})+ \int_ Q |\nabla \varphi||u_h- {\overline u}|dy.}
\end{array}
$$
Therefore, by \eqref{5.79AFP}
 and \eqref{5.80AFP}, one has
 \begin{equation}
 \label{5.81AFP}
 \displaystyle{\limsup_{h \to \infty} |D (w_h- u_h)|(Q)\leq 2(1-t^N).}
 \end{equation}
Similarly,
$$
\displaystyle{|D w_h|(Q \setminus \overline{Q}_t) \leq |D u_h|(Q \setminus {\overline Q}_t) + |D \overline{u}|(Q \setminus \overline{Q}_t) + \int_Q |\nabla \varphi||u_h-\overline{u}|dy,}
$$
consequently
\begin{equation}
\label{5.82AFP}
\displaystyle{\limsup_{h  \to \infty} |D w_h|(Q \setminus \overline{Q}_t) \leq 2 (1-t^N).}
\end{equation}

Setting $c_h:= \frac{|Du|(Q(x_0,\varrho_h))}{\varrho_h^N}$, by the scaling properties of ${\overline J_p}$ in Proposition \ref{scalingproperties} and by the growth conditions $(H_1)_p$, we have
$$
\begin{array}{ll}
\displaystyle{\frac{\overline{J}_p(u,v;Q(x_0,\varrho_h))}{|Du|(Q(x_0,\varrho_h))} = \frac{\overline {J}_p(c_h u_h, v_h; Q)}{c_h} \geq \frac{\overline{J}_p(c_h w_h, v_h;\overline{Q}_t )}{c_h}}\\
\\
\displaystyle{\geq \frac{\overline{J}_p(c_h u_h, v_h; Q)}{c_h} - C(c_h^{-1} |Q \setminus \overline{Q}_t| + |D w_h|(Q \setminus {\overline Q}_t) + c_h^{-1}\int_{Q \setminus \overline{Q}_t} |v_h|^p dy).}
\end{array}
$$
By \eqref{Lebpointsingular}, $c_h \to \infty$, moreover taking into account \eqref{vrho} and \eqref{0derivativev}, by \eqref{5.82AFP}, it results that
$$
\displaystyle{\lim_{\varrho \to 0^+}\frac{\overline{J}_p(u,v;Q(x_0,\varrho))}{|Du|(Q(x_0,\varrho))} \geq \limsup_{h \to \infty} \frac{\overline{J}_p(c_h u_h, v_h; Q)}{c_h} - 2 C(1-t^N)}.
$$
On the other hand, Lemma \ref{Lemma5.50AFP} entails that, for every $h \in \mathbb N$,
$$
\displaystyle{\overline{J}_p(c_hw_h, v_h; Q)\geq f\left(\int_Q v_hdy,c_h D w_h(Q) \right)\geq f\left(\int_Q v_h dy, c_h D u_h(Q) \right)- c_h \gamma |D(w_h- u_h)|(Q),}
$$
where $\gamma$ is the constant appearing in \eqref{p-lipschitzcontinuity}.
Then by \eqref{5.81AFP}, we have that 
$$
\displaystyle{\lim_{\varrho \to 0^+} \frac{\overline{J}_p(u,v;Q(x_0,\varrho))}{|D u|(Q(x_0,\varrho))} \geq \limsup_{h \to \infty} \frac{f\left( \int_Q v_h dy,c_h D u_h(Q)\right)}{c_h}- 2(C+ \gamma)(1-t^N).}
$$
By the definition of $u_h$, $Du_h(Q)= \frac{D u(Q(x_0,\varrho_h))}{|Du|(Q(x_0,\varrho_h))}$, hence $Du_h(Q) \to z(x_0)$, since $x_0$ is a Lebesgue point of $z$.
Now, taking into account \eqref{p-lipschitzcontinuity} and $(H_2)_p$ , we have
$$
\displaystyle{ \limsup_{h \to \infty} \frac{f\left(\int_Q v_h dy, c_h D u_h(Q)\right)}{c_h}=\lim_{h \to \infty} \frac{f\left( \int_Q v_h dy, c_h z(x_0)\right)}{c_h}}
$$
$$
\displaystyle{= \lim_{h \to \infty} \left(f^\infty\left(\int_Q v_h dy,z(x_0)\right) + C \frac{\left|\int_Q v_hdy\right|^p +1}{c_h}\right)=f^\infty(z(x_0)),} 
$$
where it has been exploited the fact that $c_h \to \infty$, 3. of Proposition \ref{proprecession}, the nondecreasing behaviour of the $L^p$ norm in the unit cube with respect to $p$ (i.e. $|\int_Q v_h dy |^p \leq \int_Q|v_h|^p dy$),  and \eqref{0derivativev}.  
\end{proof}

\subsection{Relaxation}\label{relaxation}
We start by observing that Theorem \ref{MainResultinfty} is contained in \cite{FKP1} under a uniform coercivity assumption.
We do not propose the proof in our setting, since it develops along the lines of Theorems \ref{maintheorem} and \ref{uppergeneralp}. 

On the other hand, several observations about Theorem \ref{MainResultinfty} are mandatory:
\begin{itemize}
\item[i)] If $f$ satisfies $(H_1)_p$ and $(H_2)_p$ then $\overline{J}_p(u,v)\leq \overline{J}_\infty(u,v)$ for every $(u,v)\in BV(\Omega;\mathbb R^n)\times L^\infty(\Omega;\mathbb R^m)$.

\item[ii)] For the reader's convenience we observe that the proof of the lower bound in Theorem \ref{MainResultinfty} develops exactly as that of  Theorem \ref{maintheorem}, using the $L^\infty$ bound on $v$  to deduce \eqref{0derivativev} and the uniform bound on $v_{\varrho}$ in \eqref{vrho}, $(H_2)_\infty$ and \eqref{infty-lipschitzcontinuity} in order to estimate $\displaystyle{\limsup_{h \to \infty} \frac{f\left(\int_Q v_h dy, c_h D u_h(Q)\right)}{c_h}}$.

Regarding the upper bound, the bulk part follows from \cite[Theorems 12 and 14]{RZ}, while for the singular part we can argue exactly as proposed in the proof of the upper bound in \cite{FKP1} just considering conditions $(H_1)_\infty$ and $(H_2)_\infty$ in place of $(H_1)_p$ and $(H_2)_p$.
\item[iii)] The above arguments remain true under assumptions \eqref{H1_1inftyCQf} and \eqref{H2_inftyCQf}. 
\end{itemize}

We are now in position to prove the upper bound for the case $BV \times L^p$, for $1<p<\infty$.
We emphasize that an alternative proof could be obtained via a truncation argument from the case $p=\infty$ as the one presented in \cite[Theorem 12]{RZ}, but we prefer the self contained argument below. 

\begin{theorem}\label{uppergeneralp}
Let $\Omega$ be a bounded open set of $\mathbb{R}^{N}~$ and let $f:\mathbb{R}^{m}\times\mathbb{R}^{n\times
N}\rightarrow\left[  0,\infty\right) $ be a continuous
function. Then, assuming that $f$ satisfies $(H_0), \left(  H_{1}\right)
_p$ and $(H_2)_p,$%
\begin{align*}
\overline{J}_p\left(  u,v\right)  \leq\int_{\Omega}f\left(v ,\nabla u \right)
dx +\int_{\Omega}f^{\infty}\left(0,\frac{dD^{s}%
u}{d\left\vert D^{s}u\right\vert }\left(  x\right)  \right)  d\left\vert
D^{s}u\right\vert \left(  x\right),
\end{align*}
for every $(u,v)\in BV(\Omega;\mathbb R^n)\times L^p(\Omega;\mathbb R^m)$.

\end{theorem}
\begin{proof}[Proof] First we observe that Proposition \ref{overlineJpvariational} entails that  $\overline{J}_p$ is a variational functional.
Thus the inequality can be proved analogously to \cite[Proposition 5.49]{AFP}. 
For what concerns the bulk part, it is enough to observe that
given $u\in BV(\Omega;\mathbb R^n)$ and $v \in L^p(\Omega;\mathbb R^m)$, taking a sequence of standard mollifiers $\{\varrho_{\varepsilon_k}\}$, where $\varepsilon_k \to 0$, it results that $\nabla u_k= \nabla u \ast \varrho_{\varepsilon_k}+ D^s u \ast \varrho_{\varepsilon_k}$, where $u_k:= u \ast \varrho_{\varepsilon_k}$. The local Lipschitz behaviour of $f$ in \eqref{p-lipschitzcontinuity} gives
$$
\displaystyle{\int_A f(v, \nabla u_k)dx \leq \int_A f(v, \nabla u \ast \varrho_{\varepsilon_k})dx + \gamma |D^s u|(I_{\varepsilon_k}(A))}
$$
for every $k \in \mathbb N$, where $I_{\varepsilon_k}(A)$ denotes the $\varepsilon_k$ neighborhood of $A$. Then if $|D^s u|(\partial A)=0$, letting $\varepsilon_k \to 0$, we obtain
$$
\overline{J}_p(u,v;A)\leq \int_A f(v,\nabla u)dx + \gamma |D^s u|(A),
$$
for every open subset $A$ of $\Omega$.
Thus we can conclude that
$$
\displaystyle{{\overline J}^a_p(u,v;B)\leq \int_B f(v(x), \nabla u(x))dx}
$$
for every $(u,v)\in BV(\Omega;\mathbb R^n) \times L^p(\Omega;\mathbb R^m)$ and $B$ Borel subset of $\Omega$.


To achieve the result, it will be enough to  show that
\begin{equation}
\nonumber
\displaystyle{\overline{J}^s_p(u,v;B)\leq \int_B f^\infty\left(0,\frac{d D^s u}{d |D^s u|}\right) d |D^s u|\; \; \hbox{ for every }B \hbox{ Borel subset of }\Omega.} 
\end{equation}
For every $\xi \in \mathbb R^{n \times N}$ and $b \in \mathbb R^m$, define the function
\begin{equation}
\nonumber
\displaystyle{g(b,\xi):=\sup_{t \geq 0}\frac{f(t^{\frac{1}{p}}b, t\xi)- f(0,0)}{t}.}
\end{equation}

\noindent It is easily seen that
$g$ is $(p,1)$-positively homogeneous, i.e.
$
tg(b,\xi) = g(t^{\frac{1}{p}}b,t\xi)
$
for every $t>0$, $(b,\xi)\in \mathbb R^m \times \mathbb R^{n\times N}$, $g$ is continuous and, since $f$ satisfies \eqref{p-lipschitzcontinuity}, $g$ inherits the same property.

Moreover, the monotonicity property of difference quotients of convex functions ensures that, whenever rank $\xi \leq 1$, $g(b,\xi)= f^\infty_p(b,\xi)$, where the latter is defined as
$$
f^\infty_p(b,\xi):=\limsup_{t\rightarrow \infty}\frac{f(t^{\frac{1}{p}}b,t\xi )}{t}.
$$
In particular  $g(0,\xi)= f^\infty(0,\xi)=f^\infty_p(0,\xi),$ whenever  rank $\xi \leq 1$.

\noindent Then for every open set $A \subset \subset \Omega$ such that $|Du|(A)=0$, defining for every $h \in \mathbb N$, $u_h:= u \ast \varrho_{\varepsilon_ h}$ and $v_h:= v$ where $\{\varrho_{\varepsilon_h}\}$ is a sequence of standard mollifiers and $\varepsilon_h \to 0$. Then $u_h \to u$ in $L^1$. Also \cite[Theorem 2.2]{AFP} entails that $|D u_h|\to |Du|$ weakly $\ast$ in $A$ and $|Du_h|(A) \to |D u|(A)$.
Thus
\begin{equation}
\label{stimacong}
\displaystyle{\overline{J}_p(u,v;A)\leq \liminf_{h \to \infty} \int_A f(v, \nabla u_h)dx \leq \limsup_{h \to \infty} \int_A f(v, 0)dx + \liminf_{h \to \infty}\int_A g(v,\nabla u_h)dx}.
\end{equation}
For what concerns the first term in the right hand side, we have that it is bounded by 
$
\int_A (1+|v|^p) dx,
$
thus taking the Radon-Nikod\'ym derivative with respect to ${|D^s u|}$ we obtain $0$.

Regarding the second term in the right hand side of \eqref{stimacong}, we have

$$
\begin{array}{ll}
\displaystyle{\liminf_{h\to \infty} \int_{A} g\left(v(x), Du \ast \varrho_h\right) dx}
\\
\displaystyle{\leq\limsup_{h\to \infty} \int_{A} g\left(0, Du \ast \varrho_h\right) dx+C \int_{A}|v(x)|^p dx+ \int_A |v(x)||Du \ast \varrho_h|^{\frac{1}{p'}}dx.}
\end{array}
$$

Taking the Radon-Nikod\'ym derivative, the last two terms disappear, since $|Du \ast \varrho_h| \to |Du| $, $|v|^p{\cal L}^N$ is singular with respect to $|D^c u|$ and the H\"{o}lder inequality can be applied, i.e.

$$
\displaystyle{\int_A |v(x)||Du \ast \varrho_h|^{\frac{1}{p'}}dx\leq   \left(\int_A|v(x)|^p dx\right)^{\frac{1}{p}}  \left(\int_A |D u\ast \varrho_n|dx\right)^{\frac{1}{p'}}}.
$$ 
Then the thesis is achieved via the same arguments as in \cite[Proposition 5.49]{AFP}.
\end{proof}

\begin{remark}
\label{otherproof}
It is worth to observe that an alternative argument to the one presented above, concerning the upper bound inequality for the singular part, can be provided by means of approximation. In fact, one can prove that $\displaystyle{\overline{J}^s_p(u,v;B)\leq \int_B f^\infty\left(0,\frac{d D^s u}{d |D^s u|}\right) d |D^s u|\; \; \hbox{ for every }B \hbox{ Borel subset of }\Omega,} $ when $u \in BV(\Omega;\mathbb R^n)$ and $v \in C({\overline \Omega};\mathbb R^m)$, and then via a standard approximation argument via mollification allows to reach every $v \in L^p(\Omega;\mathbb R^m)$.

For what concerns the case $v \in C({\overline \Omega};\mathbb R^m)$ it is enough to consider the function

\noindent$
\displaystyle{g(b,\xi):=\sup_{t \geq 0}\frac{f(b, t\xi)- f(b,0)}{t},}$ 
exploit its properties of positive 1-homogeneity in the second variable, i.e.
$
tg(b,\xi) = g(b,t\xi),
$
for every $t>0$, $(b,\xi)\in \mathbb R^m \times \mathbb R^{n\times N}$, \eqref{infty-lipschitzcontinuity}, and the fact that when rank $\xi \leq 1$, then $g(b,\xi)$ is constant with respect to $b$ and $f^\infty(b,\xi)= g(b,\xi)=f^\infty(0,\xi)$. 
To conclude it is enough to apply Reshetnyak continuity theorem.
\end{remark}

\medskip
\begin{proof}[Proof of Theorem \ref{MainResultp}]
The result follows from  Theorems \ref{maintheorem} and \ref{uppergeneralp}, applying Proposition \ref{relaxedp=relaxedp} to remove assumption $(H_0)$.
\end{proof}

\medskip

\noindent{\bf Acknowledgements}
The research of the authors has been
partially supported by Funda{\c c}${\tilde {\rm a}}$o para a Ci${\hat {\rm e}}$ncia e Tecnologia (Portuguese Foundation for Science and
Technology) through UTA-CMU/MAT/0005/2009 and CIMA-UE.

\end{document}